\documentclass[10pt,twocolumn,a4paper]{article}
\usepackage[utf8]{inputenc}
\usepackage[T1]{fontenc}
\usepackage{lmodern}
\usepackage[margin=20mm,columnsep=7mm]{geometry}
\usepackage{microtype}
\usepackage[british]{babel}
\usepackage{float}
\usepackage[mathscr]{euscript}
\usepackage{amsmath, amssymb, amsfonts,amsthm}
\usepackage{mathrsfs}
\usepackage{xcolor}
\usepackage{faktor}
\usepackage{graphicx}
\usepackage{epstopdf}
\usepackage{hyperref}

\newtheorem{proposition}{Proposition}[section]

\theoremstyle{definition}

\newtheorem{definition}{Definition}[section]

\title{Frequency Structure of AMM Liquidity Pools: A Multiresolution Analysis of Regime Transitions}
\author{Dennis N. Páez Chala\\ \small  Institut Universitari de Matemàtica Pura i Aplicada\\ \small Universitat Politècnica de València\\ \texttt{dnpaecha@posgrado.upv.es}
  \and Alfred Peris\\ \small  Institut Universitari de Matemàtica Pura i Aplicada\\ \small Universitat Politècnica de València\\ \texttt{aperis@upv.es}}
\date{}
\begin{document}

\makeatletter
\twocolumn[\begin{@twocolumnfalse}
\maketitle
\begin{abstract}
This study examines the temporal structure of liquidity pools in constant-product automated market makers. Using a sample of eighteen pools observed at a five-minute frequency over two years, it aims to identify the temporal scales at which interactions within the system occur and how these interactions reorganize in response to changes in the volatility regime. Methodologically, the study uses a past-looking Haar multiresolution decomposition, so that the wavelet features at each date depend only on observations available by that date. Regime labels additionally depend on the parameter-estimation protocol used for the Markov model. Within this decomposition, three dimensions are examined: the scale composition of the return-to-risk relationship, the frequency-band coupling between pool activity and price volatility, and the detection of regime transitions using a fluctuation-to-trend indicator. The results reveal a clear structural segmentation: pool variability is concentrated at fast scales, whereas the relationships among its series, signal persistence, and information on subsequent performance reside at slow scales. Aggregate measures do not distinguish this separation because they combine both components before they can be compared.
\end{abstract}
\vspace{1.5em}
\end{@twocolumnfalse}]
\makeatother

\section{Introduction}

This study examines the frequency structure of liquidity-pool time series\footnote{In the decentralised finance (DeFi) ecosystem, a liquidity pool is a smart contract that holds reserves of two or more assets to enable automated exchanges without intermediaries. This mechanism, governed by algorithms known as automated market makers (AMMs), replaces the traditional order book used in centralised markets with a mathematical curve that determines the execution price according to the remaining reserves.}: how their variability is distributed across temporal scales and what this distribution reveals when the market changes state. The objective is to determine whether decomposition by scale, through multiresolution analysis (MRA), reveals a structure that aggregate analysis combines or fails to identify: 
%\AP{A continuación pones ``both'', pero aparecen 3, y con 2 ``and''}
%\NP{Corregido}
structure in the relationships among the pool's time series, in the quality of the signals produced by its indicators and in the information those signals contain about its future behaviour. As an overarching constraint, the multiresolution features are constructed causally, using observations no later than the date they describe. Strict real-time interpretation of estimated regime labels additionally requires a historical parameter-estimation scheme.

Scale-based analysis requires frequency-domain tools. The classical approach, based on the Fourier transform, presupposes, for its interpretation, that the statistical properties of a signal do not change over time \cite{percival2000}. However, financial time series characteristically violate this condition: their level, variance, and internal relationships change continuously, which defines them as non-stationary processes. \textit{Wavelets} overcome this limitation by localising frequency analysis simultaneously in time and scale \cite{crowley2007}, thereby making it possible to describe which variations dominate the behaviour of the pool at each point in time.

Nevertheless, the characterisation of these frequencies becomes analytically meaningful only when considered in relation to the market context. Therefore, the notion of a change of state---understood as the alternation between persistent periods of calm and turbulence---requires its own frame of reference. To label these transitions, we adopt Hamilton's Markov regime-switching model \cite{hamilton1989new}, specified in this study using two states. This framework is selected because the model identifies the starting dates of turbulent episodes from the data and is independent of the \textit{wavelet} structure. In this way, the definition of the event and the instrument used to analyse it remain separate.

Having presented the two tools, it remains to introduce the system to which they are applied, which originates in decentralised finance. In an automated market maker (AMM), the exchange price is determined by a formula applied to the reserves deposited in a contract. The liquidity pool is this deposit: liquidity providers contribute the two assets of a pair to facilitate transactions in exchange for the fees generated by each transaction, while remaining exposed to price volatility \cite{milionis2022lvr}.

There are two main generations of these pools. In V2 pools, all liquidity providers share the same position across the entire price range \cite{adams2020uniswapv2}, whereas, in V3 pools, each provider selects an individual range \cite{adams2021uniswapv3}. This study focuses on V2 pools for a methodological identification reason: the homogeneity of their positions ensures that the observable pool time series describe a single common system, without the additional dimension introduced by individual range choices.

With the system and the analytical tools defined, the methodological proposal consists of decomposing the pool's time series using a causal \textit{wavelet} operator and analysing the resulting components in three directions: first, measuring the energy of each band of scales to determine how much variability it carries at each point in time; second, correlating these energies band by band to identify the specific scales at which pool activity and price volatility become synchronised; and third, classifying the signals produced by the pool indicators according to the temporal scale from which they originate. The question therefore shifts from how much the series vary to the speed at which each phenomenon occurs.

Regarding the data, the analysis uses five-minute time series---price, volume, and deposited capital---from V2 pools over a two-year period. Data processing is implemented in Python, while the analysis and figures are produced in MATLAB.

The remainder of this study is organised as follows. Section 2 presents the literature review. Section 3 introduces the foundations of multiresolution analysis. Section 4 details the methodology: the \textit{wavelet} transform, the causal operator, band energies, scale-wise correlation, and the regime model. Section 5 describes the structure of the problem, where the liquidity-pool variables are formalised. Section 6 presents the data and details the sources and the selected panel of pools. Finally, Section 7 presents the results, and Section 8 provides the conclusions of the study.

%-----------------------------------

\section{Literature Review}

In traditional finance, Hamilton's Markov regime-switching model \cite{hamilton1989new} provides the standard formalisation for the alternation between persistent calm and turbulent states. Later refinements allow transition probabilities to depend on duration \cite{durland1994} or on a vector of time-varying observable variables $z_t$ \cite{diebold1994}, while empirical evidence confirms that a two-state specification is sufficient to capture equity returns \cite{angbekaert2002}. Here, the two-state model serves solely to identify regime transitions, leaving the frequency structure around them to be evaluated using a separate framework.

In this context, decentralised finance not only captures these classical dynamics of alternation but also makes the variables describing the state of the system publicly available at the same temporal resolution as the price. This availability allows the vector \(z_t\) to cease depending on approximations estimated externally and instead be constructed from direct measurements of the market itself. Consequently, it is possible to investigate not only when a regime transition occurs, but also the specific temporal scale at which this change manifests itself in each of these state variables \cite{adams2020uniswapv2}.

The question of the temporal scales at which assets concentrate their fluctuations is examined by Gençay, Selçuk, and Whitcher \cite{gencay2002introduction}, who formalise \textit{wavelet} analysis as an econometric tool adapted to the non-stationary nature of financial time series. Crowley \cite{crowley2007}, in turn, structures the economic questions addressed by this multiresolution analysis. Thus, the existing literature supports the methodological shift introduced in this study: moving the focus away from the aggregate variability of the signal and towards identifying the specific temporal scales at which its dynamics operate.

The feasibility of applying this decomposition to non-stationary series lies in the localisation properties of the \textit{wavelet} basis. Unlike Fourier trigonometric functions, which have global support and distribute any local event across all coefficients, \textit{wavelets} have compact support. Since they are generated through translations and dilations in time and scale, a shock affects only the coefficients whose supports overlap with it. Because this basis is orthogonal and complete \cite{daubechies1992}, the original series can be reconstructed exactly, providing the advantage of identifying simultaneously when and at what scale each fluctuation occurs.

The application of this decomposition to financial series has two precedents aimed at isolating short-term fluctuations from the underlying price trend. In the first, Beltrani \cite{beltrani2022crashes} developed a ratio indicator to measure rapid volatility relative to the general trend and applied it to United States stock market indices, identifying six clusters of dates that coincide with different historical crises.

Vicente Arévalo \cite{vicente2026}, in turn, evaluates the sectors of the S\&P 500 using an indicator that relates the average of the absolute differences between consecutive detail coefficients to the mean value of the approximation at the same level, thereby measuring the variability of the detail fluctuation at a single temporal scale. Although the indicator does not always detect the onset of crises, the study shows that returns stabilise rapidly and become positive again after each event has passed.

In the field of cryptoassets, scale-based analysis has revealed temporal dynamics that aggregate methods fail to capture. On the one hand, Kristoufek \cite{kristoufek2015main} demonstrates that the explanatory factors of the Bitcoin price operate across different frequency bands and vary according to the market phase. Using a causal approach, Alvarez-Ramirez, Espinosa-Paredes, and Vernon-Carter \cite{alvarez2025causal} implement a \textit{wavelet} analysis that uses only past information and find a contemporaneous coincidence between a greater concentration of energy at certain scales and Bitcoin bear markets. Finally, Omane-Adjepong, Alagidede, and Akosah \cite{omane2019wavelet} analyse persistence and efficiency in these markets through \textit{wavelets} across different regimes and show that volatility persistence is highly sensitive to both the temporal scale and the state of the market.

By contrast, multiscale analysis of automated market makers (AMMs) remains at an early stage. A relevant precedent is the study by Watorek and collaborators \cite{wkatorek2024approaching}, who find that the complexity of decentralised trading is not constant but emerges during large fluctuations and manifests itself more strongly in volume than in returns. The remaining literature has been limited to studying price properties through traditional approaches, without analysing the pool as a system or the distribution of its variability across scales, leaving a methodological gap that this study seeks to address.

%-----------------------------------
\section{Theoretical Framework}

The multiresolution analysis (MRA) considered in the introduction can be described as a hierarchy of vector spaces designed to represent functions at different levels of precision; it is, in itself, a formal mathematical framework for decomposing functions and signals across multiple resolution scales \cite{bookMPR2004,walker2008primer}.

This section develops that framework in three stages: the Haar transform, the discrete \textit{wavelet} transform (DWT) derived from it, and the resulting multiresolution analysis (MRA). On this basis, the subsequent methodology section adapts this structure to present the specific transform variant and its causal formulation.

In particular, discrete series \(f\) expressed as
\[
f=(f_1,f_2,\ldots,f_N)\in\mathbb{R}^N,
\]
are considered, where \(N\) is the number of observations in the series and \(f_n\) is the value observed at time \(n\).

%\AP{Lo que viene a continuación sobre el MRA no coincide con lo visto en clase. Al final pongo alternativa, para adaptar}

%\NP{Modifique la sección, saque todo lo referente a la transformada continua, modifique la estructura para que el marco teórico incluya Haar, DWT  y MRA  como el resultado de la construcción, adaptando el contenido que me pasaste. En metodología deje MODWT y lo que es bien especifico del trabajo}

\subsection{Haar Transform}

The Haar \textit{wavelet} is characterised by a stepwise structure. It is constructed from two fundamental vectors: the mother \textit{wavelet} and the scaling vector \cite{bookMPR2004,walker2008primer}.

The mother \textit{wavelet} represents the high-frequency details or rapid fluctuations of the signal.

\begin{definition}[Haar mother wavelet]
\label{def:wavelet-madre-haar}

\[
\psi=\left(\tfrac{1}{\sqrt{2}},-\tfrac{1}{\sqrt{2}}\right)
\]
\end{definition}

This vector has zero sum,
\[
\sum_i\psi_i=\tfrac{1}{\sqrt{2}}-\tfrac{1}{\sqrt{2}}=0,
\]
and unit energy,
\[
\lVert\psi\rVert^2=\tfrac12+\tfrac12=1
\]

The zero-sum property ensures that \(\psi\) ignores the mean value of the signal and captures only its changes, while unit energy normalises the magnitude of the coefficients.

The scaling vector, in turn, represents the approximation of the signal.

\begin{definition}[Haar scaling vector]
\label{def:wavelet-padre-haar}

\[
\phi=\left(\tfrac{1}{\sqrt{2}},\tfrac{1}{\sqrt{2}}\right)
\]
\end{definition}

This scaling vector also has unit energy and is strictly orthogonal to the mother \textit{wavelet} (\(\phi\cdot\psi=0\)). When applied to a pair of consecutive observations \((f_{2k-1},f_{2k})\), the vector \(\phi\) calculates their normalised average, while \(\psi\) yields their normalised difference. Together, they form the elementary unit of the subsequent construction: the functions at the first level of the discrete \textit{wavelet} transform (DWT) are obtained by translating \(\phi\) and \(\psi\) along the signal, while higher levels dilate them by successively doubling their temporal support.

%% TODO-NP: el segundo párrafo de la sección («This section develops that framework in three stages…») y la última frase de 3.1 («Together, they form the elementary unit of the subsequent construction…») mencionan la DWT, que ya no está: reescribir las dos.

\subsection{Multiresolution Analysis (MRA)}

%% TODO-NP: esta frase remitía a la DWT que ya no está; reescribir el arranque (p. ej. que la aplicación recursiva de los vectores de Haar hasta el nivel J da la descomposición siguiente, citando \cite{bookMPR2004,walker2008primer}).
The recursive application of this transform leads to the fundamental decomposition of multiresolution analysis: at level \(J\) (with \(J\leq K\)), every signal is expressed as
\[
f = A_J + D_J + D_{J-1} + \cdots + D_1,
\]

Here, \(A_J\) represents the approximation at level \(J\) (the trend that remains after \(J\) successive averaging operations), and each \(D_j\) is the detail component at level \(j\) (the fluctuation lost when moving from resolution level \(j-1\) to level \(j\)).

%% TODO-NP: el párrafo siguiente usa los coeficientes a_{J,k} y d_{j,k}, que se definían en la DWT recortada; o agregar media frase que los presente, o dejar solo la frase final sobre Parseval sin la fórmula.
The orthonormality of the Haar systems has a direct consequence that anticipates the central role of energy in this study: The total energy of the signal is preserved, although its distribution across the different scales yields a phenomenon of compactification of energy, since usually the approximation coefficients are much greater than the detail coefficients:
\[
\sum_{n=1}^{N} f_n^2 = \sum_{k=1}^{N/2^J} a_{J,k}^2 + \sum_{j=1}^{J} \sum_{k=1}^{N/2^j} d_{j,k}^2.
\]
The first term on the right-hand side represents the energy of the  approximation sub-signal \(a_J\), and each term in the double sum corresponds to the energy of the detail components \(d_j\) at level \(j\). This identity, the discrete equivalent of Parseval's theorem \cite{bookMPR2004,walker2008primer}, provides the formal basis for analysing how much variability is concentrated at each temporal scale.

% --------------
% --------------
% --------------
%%-------------------------
\section{Methodology}

\subsection{Maximum Overlap Discrete Wavelet Transform (MODWT)}

For the MODWT we directly consider general length \(N\), and we will see that it retains the same number of coefficients at each decomposition scale, in contrast with the DWT. This property makes it translation-invariant and produces coefficients that are temporally aligned with the original series \cite{percival2000}. This temporal alignment motivates its selection for the analysis of the signals under study.

To compensate for the energy redundancy caused by the absence of subsampling, the Haar filters in the MODWT are normalised by dividing the corresponding conventional DWT filters by an additional factor of \(\sqrt{2}\). This normalisation yields the scaling and detail filters defined below.

\begin{definition}[Scaling filter]
\label{def:haar-modwt-scaling-filter}
For the Haar MODWT, the filter coefficients are
\[
h=
\left(
\frac{1}{2},
\frac{1}{2}
\right)
\]
\end{definition}

\begin{definition}[Detail filter]
\label{def:haar-modwt-detail-filter}
For the Haar MODWT, the filter coefficients are
\[
g=
\left(
\frac{1}{2},
-\frac{1}{2}
\right)
\]
\end{definition}

\subsubsection{Approximation and Detail Coefficients}

For each scale \(j\geq 1\), both coefficients are calculated at each time index \(n\) by means of circular convolution. In this process, the corresponding filters are progressively dilated by inserting intermediate zeros with a dyadic spacing of \(2^{j-1}\), following the \textit{à trous} algorithmic structure \cite{holschneider1989}.

The initial condition and periodic indexing convention are incorporated into the following definitions.

\begin{definition}[Approximation coefficient]
\label{def:haar-modwt-approximation-coefficient}
Let
\[
A_{0,n}=f_n,
\qquad n=1,\ldots,N,
\]
and interpret the indices periodically modulo \(N\). For every \(j\geq 1\) and \(n\in\{1,\ldots,N\}\), define
\[
A_{j,n}
=
\frac{
A_{j-1,n}
+
A_{j-1,n-2^{j-1}}
}{2}
\]
\end{definition}

\begin{definition}[Detail coefficient]
\label{def:haar-modwt-detail-coefficient}
Let \(A_{j-1,n}\) denote the approximation coefficients from Definition~\ref{def:haar-modwt-approximation-coefficient}, with indices interpreted periodically modulo \(N\). For every \(j\geq 1\) and \(n\in\{1,\ldots,N\}\), define
\[
D_{j,n}
=
\frac{
A_{j-1,n}
-
A_{j-1,n-2^{j-1}}
}{2}
\]
\end{definition}

Both sets of coefficients have an intuitive arithmetic interpretation. The approximation coefficients \(A_{j,n}\) represent the overall structure, trend, and low-frequency components of the signal, whereas the detail coefficients \(D_{j,n}\) capture rapid fluctuations, noise, and high-frequency components.

As a result, the coefficients remain temporally aligned with each observation in the series under study. To adapt the transform to the objectives of this article, the following section introduces a strictly causal version in which the filtering windows extend only backwards in time rather than relying on periodic indexing.

\subsubsection{Energy Conservation}

With periodic indexing and the stated Haar normalisation, the non-decimated transform has an exact global energy partition. This statement concerns the complete periodic sequence, not arbitrary truncated windows of the causal implementation.

\begin{proposition}
\label{prop:conservacion-energia-modwt}
Let \(f\in\mathbb{R}^N\), use periodic indexing modulo \(N\), and define \(A_{j,n}\) and \(D_{j,n}\) by Definitions 4.3--4.4. For every integer \(J\geq1\), these periodic Haar coefficients satisfy
\[
\sum_{n=1}^{N} f_n^2
=
\sum_{j=1}^{J}
\left(
\sum_{n=1}^{N} D_{j,n}^2
\right)
+
\sum_{n=1}^{N} A_{J,n}^2
\]
\end{proposition}

This is the periodic MODWT energy identity \cite{percival2000}. Indeed, writing \(S_h\) for circular translation by \(h=2^{j-1}\), one has
\[
 A_j=\tfrac12(I+S_h)A_{j-1},\qquad
 D_j=\tfrac12(I-S_h)A_{j-1}.
\]
Since \(S_h\) is an isometry, the parallelogram identity gives
\(\|A_j\|_2^2+\|D_j\|_2^2=\|A_{j-1}\|_2^2\); summing over levels proves the proposition. This does \emph{not} assert mutual orthogonality of all reconstructed components. For a finite causal series without periodic extension, the equality cannot automatically be applied to a truncated observation interval because boundary terms and missing initial coefficients intervene. Consequently, local band energies below are descriptive squared-coefficient measures, not an exact additive partition of local signal variance.

%----------------------------

\subsection{Haar \textit{à trous} operator}

Using the theory and definitions introduced in the previous section, we adopt a single Haar \textit{à trous} operator \cite{renaud2003prediction}, whose formal definition is presented below.

\begin{definition}[Haar \textit{à trous} operator]
\label{def:operador-haar-a-trous}
Let \(x(t)\) be a discrete time series, and let \(J\geq 1\) be the decomposition depth. For each \(j=1,\ldots,J\), define
\[
A_j(t)
=
\frac{1}{2^j}
\sum_{k=0}^{2^j-1}x(t-k),
\qquad
D_j(t)
=
A_{j-1}(t)-A_j(t),
\]
with the initial condition
\[
A_0(t)=x(t)
\]

\end{definition}

The function \(A_j(t)\) represents the approximation at scale \(j\), defined as the average of the most recent \(2^j\) observations, whereas \(D_j(t)\) represents the detail at scale \(j\), determined by the difference between two consecutive approximations.

This operator is characterised by two fundamental analytical properties. First, the original series can be recovered exactly from its decomposition, as formalised below:

\begin{proposition}[Exact reconstruction]
\label{prop:reconstruccion-exacta}
For every series \(x\) and every decomposition depth \(J\geq 1\), the following identity holds at every time  \(t\geq 2^J\):
\[
x(t) = A_J(t) + \sum_{j=1}^{J} D_j(t)
\]
\end{proposition}

\begin{proof}
The sum of the detail coefficients reduces algebraically to a telescoping sum:
\[
\begin{aligned}
\sum_{j=1}^{J} D_j(t)
&=
\sum_{j=1}^{J}
\left[
A_{j-1}(t)-A_j(t)
\right]
\\
&=
A_0(t)-A_J(t)
=
x(t)-A_J(t)
\end{aligned}
\]
\end{proof}

%\AP{Aqui también vendría bien algún ejemplo gráfico}

Second, causality is mathematically guaranteed by construction, since the approximation coefficient is defined using only the historical terms \(x(t-k)\) for \(k\geq 0\). Thus, at each time \(t\), the calculation of the coefficients is restricted to the temporal support
\[
\left[t-2^j+1,t\right],
\]
ensuring that the operator processes exclusively the information available up to that time.

To illustrate how this restriction operates when the operator is applied to a finite signal, consider a transient signal previously used as an example in the literature \cite{bookMPR2004}:
\begin{equation}
\label{eq:scale-signal}
f(t)=20t^2(1-t)^4\cos(12\pi t).
\end{equation}
The signal is sampled at \(N=128\) equally spaced points over \([0,1]\).

\begin{figure}[H]
\centering
\includegraphics[width=\columnwidth]{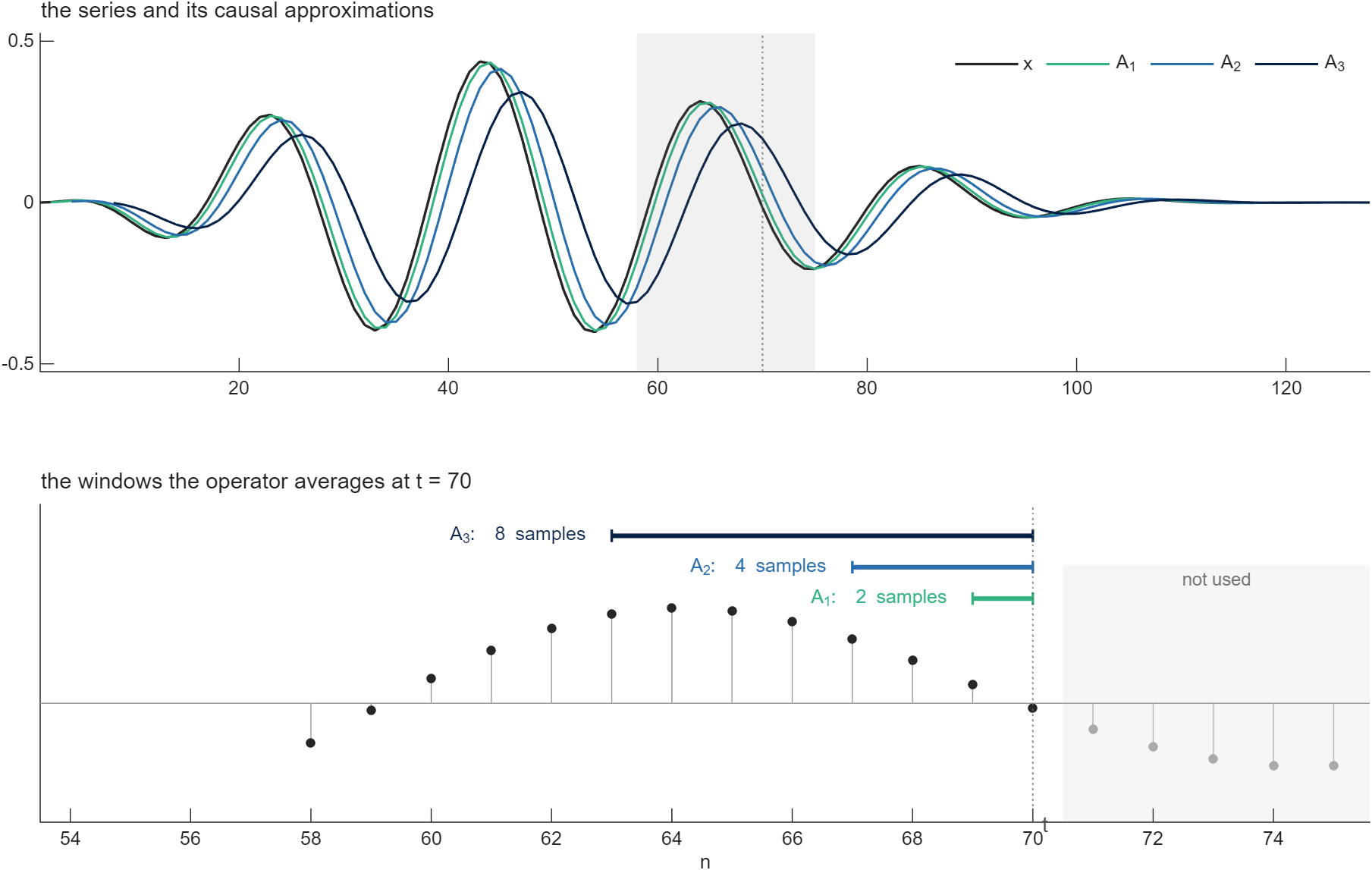}
\caption{Causal \textit{à trous} Haar operator applied to the signal~\eqref{eq:scale-signal}. The upper panel compares the signal \(x\) with the approximations \(A_1\), \(A_2\), and \(A_3\); as the depth increases, so do the smoothing and temporal lag. The lower panel enlarges the shaded region and shows the windows \([t-2^j+1,t]\) used at \(t=70\). All bars end at \(t\), so subsequent observations, shown in grey, are excluded.}
\label{fig:atrous-operator}
\end{figure}

\subsubsection{Depth and Scales}

In connection with the causal structure, the definition of the scales \(j\) and the depth \(J\) determines the temporal resolution of the analysis and the magnitude of the initial boundary effect. Away from the initial boundary, the causal coefficients agree with the corresponding past-looking periodic Haar MODWT coefficients whenever the full historical support is available. Each scale is therefore associated with a nominal frequency band \cite{percival2000}.

In this sense, the detail \(D_j\) is primarily associated with nominal periods from \(2^{j}\) to \(2^{j+1}\) sampling intervals, whereas \(A_J\) emphasises slower dynamics. These are approximate octave labels rather than disjoint spectral supports: Haar filter frequency responses overlap and have non-zero leakage outside their nominal bands. Therefore, translating these results into specific time units depends on the sampling grid of the series, an aspect that will be detailed when formalising the model variables.

With regard to the depth \(J\), it determines an analytical trade-off: each additional scale isolates a lower-frequency octave but doubles the size of the backward-looking filtering window, thereby further reducing the number of available observations. Therefore, the criterion consists of selecting the smallest value of \(J\) that ensures that the final approximation \(A_J\) retains only the long-term trend component without absorbing scale-specific fluctuations, a choice that will be justified in the results section.

\subsection{Band Energies}

While the energy-conservation proposition distributes the total energy of the signal among the different scales, analysing its temporal evolution requires this metric to be localised in time. For this purpose, the energy measure is defined over a moving window, which is consistent with the causality of the operator.

\begin{definition}[Local scale and band energy]
\label{def:energia-local-escala-banda}
Let \(x\) be the series, let \(j\) be the scale, and let \(D_jx\) be the detail sequence corresponding to that scale. For \(w\in\mathbb{N}^{+}\), the causal window of width \(w\) associated with time \(t\) is defined as
\[
I_t=\{t-w+1,\ldots,t\}
\]

The local energy at scale \(j\) and time \(t\) is defined as the squared \(\ell^2\)-norm of the detail sequence restricted to \(I_t\):
\[
E_j(t)
=
\sum_{s\in I_t}(D_jx)_s^2
=
\left\|
\left.(D_jx)\right|_{I_t}
\right\|_{\ell^2}^{2}
\]

Let \(B\) be a band formed by a set of adjacent scales. The energy of band \(B\) at time \(t\) is defined as
\[
E_B(t)
=
\sum_{j\in B}E_j(t)
\]
\end{definition}

Dividing the local energy \(E_j(t)\) by the moving-window width \(w\) yields the quantity \(\frac{1}{w}E_j(t)\), which constitutes the sample estimator of the \textit{wavelet} variance at scale \(j\) \cite{percival2000}. The resulting local squared-coefficient statistic is related to wavelet-variance estimation, but its finite-window bias, dependence and boundary effects require separate consideration; properties established for global estimators do not transfer automatically.

%\AP{Añadir ejemplo ilustrativo}
%\NP{Añadido}

Figure~\ref{fig:haar-3} illustrates this decomposition using the signal defined in equation~\eqref{eq:scale-signal}, decomposed to a depth of \(J=3\).

\begin{figure}[htbp]
\centering
\includegraphics[width=\columnwidth]{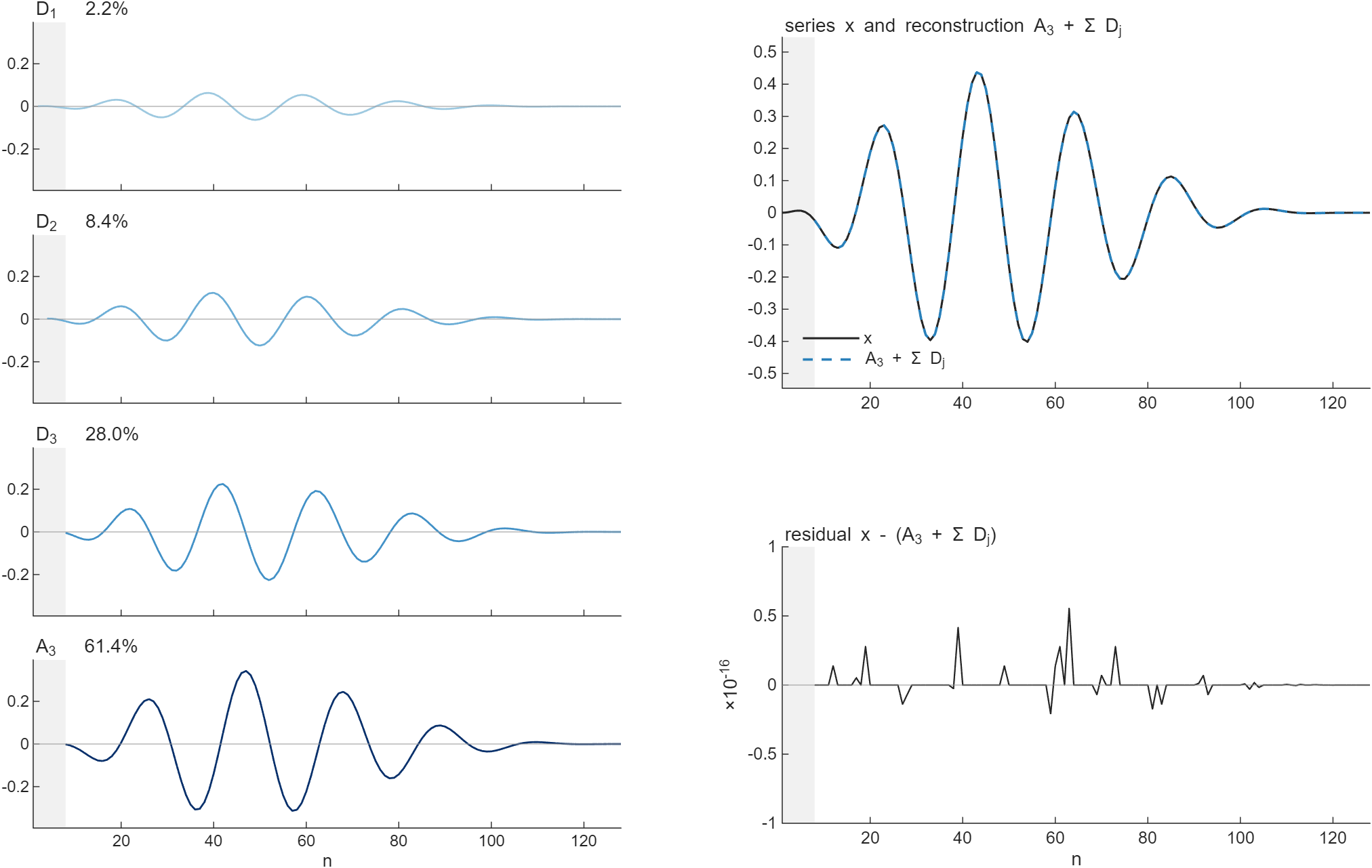}
\caption{Causal \textit{à trous} Haar decomposition of the signal~\eqref{eq:scale-signal} with depth \(J=3\). The left panel shows the detail components \(D_1\), \(D_2\), and \(D_3\), together with the approximation \(A_3\), all temporally aligned with the original signal. The right panel compares the signal with its reconstruction \(A_3+\sum_{j=1}^{3}D_j\). The shaded region marks the initial interval \(t<2^J\), where \(2^J=8\), in which \(A_3\) is not yet defined.}
\label{fig:haar-3}
\end{figure}

\subsection{Scale-wise Correlation}

Local energy examines one signal at a time; however, our study requires tools for analysing two signals. For this purpose, wavelet correlation is the instrument used to compare two series across scales. It measures their synchronised movement and its direction by directly relating their detail coefficients scale by scale \cite{whitcher2000wavelet}.

However, the research question addressed in this study is different, as we seek to determine whether the intensity of the variability of one series accompanies that of another at each scale, regardless of its direction. For this purpose, the appropriate tool is the correlation between the previously defined band energies, calculated over a historical window.

\begin{definition}[Band-energy correlation]
\label{def:correlacion-energias-banda}
Let \(x\) and \(y\) be two time series defined on the same sampling grid, let \(B\) be a band of scales, and let
\(
I_t=\{t-W+1,\ldots,t\}
\)
be the causal window of width \(W\in\mathbb{N}^{+}\). The band-energy correlation at time \(t\) is defined as
\[
\rho_B(t)
=
\operatorname{corr}_{\mathrm{S}}
\left\{
\left(E_B^x(s),E_B^y(s)\right)
:
s\in I_t
\right\}
\]

Here, \(\operatorname{corr}_{\mathrm{S}}\) denotes Spearman's correlation.
\end{definition}

%\AP{Eso que ``The width \(W\) of this correlation window is independent of the width \(w\)'' no se acaba de entender muy bien}
%\NP{Añadi el parrafo siguiente para explicarlo bien}

This construction uses two time windows with distinct roles. The inner window, of width \(w\), defines how much recent history each energy value \(E_B^{x}(s)\) captures. The outer window, of width \(W\), determines how many of these values are used to calculate Spearman's correlation at time \(t\). The two parameters are chosen separately: \(w\) controls the degree of smoothing of the energy, while \(W\) determines the sample size of the correlation estimator. Because consecutive inner and outer windows overlap, the resulting correlation series is serially dependent. Inference on event-time differences should therefore use uncertainty estimates that preserve this dependence (for example, a block bootstrap clustered by pool and turbulence episode); a numerical difference between correlations alone is not a significance test.

A rank-based metric is selected because the local energies are, by construction, non-negative and strongly skewed. Given this structure, Spearman's correlation \cite{spearman1904proof} ensures that extreme values enter only through their ordinal position rather than their absolute magnitude; therefore, their influence on the estimator remains bounded. Furthermore, this property makes the metric invariant under strictly monotonic transformations of the series and allows it to capture any monotonic relationship between them.

%------------------------------------
\subsection{Regimes: The Markov Model}

To study how scale-dependent dynamics are restructured when the market changes state, an operational definition of that state is required. This definition must be external to the \textit{wavelet} formulation, since identifying the episodes using the same instruments that will subsequently be employed to analyse them would introduce methodological circularity. Thus, the regime model delimits periods of analysis whose construction does not involve any multiresolution tool.

\begin{definition}[Markov regime-switching model]
\label{def:modelo-regimen-markov}
Let \(P_d\) be the closing price of the pool on day \(d\). The daily logarithmic return is defined as
\[
r_d
=
\log(P_d)-\log(P_{d-1}),
\]
and its standardised version as
\[
z_d
=
\frac{r_d-\overline{r}}{s_r},
\]
where \(\overline{r}\) and \(s_r\) represent the mean and standard deviation used to standardise daily logarithmic returns. Their estimation window must be specified when interpreting the labels as real-time outputs.

%\AP{Arriba es ``daily return series'' o ``daily logarithmic return series''?}
%\NP{Corregido}

Let \(S_d\in\{0,1\}\) be a latent variable that follows a first-order Markov chain, whose transition probabilities are defined as
\[
p_{ij}
=
\mathbb{P}
\left(
S_d=j
\mid
S_{d-1}=i
\right),
\qquad
i,j\in\{0,1\}.
\]
Conditional on the latent state, the standardised series satisfies
\[
z_d\mid S_d=i
\sim
\mathcal{N}
\left(
\mu,\sigma_i^2
\right),
\qquad
i\in\{0,1\}.
\]
\end{definition}

The parameters \(\mu\), \(\sigma_0^2\), \(\sigma_1^2\), \(p_{00}\), and \(p_{11}\) are estimated by maximum likelihood. Once the model has been estimated, the turbulent state is identified as the state with the largest variance:
\[
i_{\mathrm{turb}}
=
\underset{i\in\{0,1\}}{\arg\max}\,
\sigma_i^2.
\]

Once the parameters have been estimated, the Hamilton filter calculates, for each day, the filtered probability of belonging to each state \cite{hamilton1989new}. This calculation uses only the observations available up to that time:
\[
\mathbb{P}
\left(
S_d=i
\mid
\mathcal{F}_d
\right),
\qquad
i\in\{0,1\},
\]
where \(\mathcal{F}_d\) contains the observations available up to day \(d\). Operationally, a day is classified as turbulent when
\[
\mathbb{P}
\left(
S_d=i_{\mathrm{turb}}
\mid
\mathcal{F}_d
\right)
>
0.5.
\]

The Hamilton recursion filters the state conditional on the chosen model parameters. A filtered probability at day \(d\) is strictly real-time only if the standardisation constants and model parameters have themselves been estimated from observations available by day \(d\) (for example, by an expanding-window or rolling-window fit). If instead they are fitted once using the complete sample, the labels are retrospective, despite the use of filtered rather than smoothed probabilities. The present event study should be interpreted as retrospective unless such historical re-estimation is verified. In either case, the regime model does not use the wavelet decomposition.

This segmentation establishes the temporal framework within which the redistribution of energy across scales will be analysed. Specifically, transitions to the turbulent regime will serve as a reference for examining the behaviour of the signals around each change of state. Finally, in this study, the model is used exclusively to determine the starting dates of turbulent episodes.

\subsection{Fluctuation-to-trend indicator (\(H_i^c\))}

While band energies quantify the variability of a series at each scale, a second interpretation of the decomposition consists of relating the accumulated fluctuation of the details to the level of the trend at the corresponding resolution.

\begin{definition}[Causal indicator \(H_i^c\)]
\label{def:indicador-causal-hc}
Let \(x\) be a strictly positive time series in levels and \(i\in\mathbb{N}^{+}\) a decomposition depth, with at least \(2^i\) historical observations available. For each time \(t\), define
\begin{equation}
H_i^c(t)
=
\frac{\displaystyle\sum_{j=1}^{i}\left|D_j(t)\right|}
{\left|\widetilde{A}_i(t)\right|},
\label{eq:indicador-causal-hc}
\end{equation}
where \(D_j(t)\) are the detail coefficients of the decomposition and \(\widetilde{A}_i(t)\) is the median of the \(2^i\) most recent observations, \(t\) included, which replaces the approximation \(A_i(t)\) as an estimate of the local reference level. Strict positivity ensures that the denominator is non-zero.
\end{definition}

%\AP{Poner un ejemplo}
%\NP{Añadido}

The calculation of the indicator can be followed step by step over a short segment of the series. For \(i=2\), suppose that the four most recent observations, including the observation at time \(t\), are \((100,102,106,112)\). The indicator is obtained as follows:

\begin{enumerate}
\item Select the \(2^i=4\) most recent observations, which constitute the support of the operator at this depth.

\item Calculate the causal approximations: \(A_1(t)=109\), corresponding to the mean of the last two observations, and \(A_2(t)=105\), corresponding to the mean of all four observations.

\item Obtain the detail coefficients as differences between consecutive approximations: \(D_1(t)=x(t)-A_1(t)=3\) and \(D_2(t)=A_1(t)-A_2(t)=4\). Therefore, the accumulated fluctuation is \(|D_1(t)|+|D_2(t)|=7\).

\item Divide this value by the median of the four observations, \(\widetilde{A}_2(t)=104\), which yields \(H_2^c(t)=7/104\approx 0.067\).
\end{enumerate}

If the four observations were \((100,101,100,101)\), oscillating around the same level, both approximations would be equal to \(100.5\), and the accumulated fluctuation would be \(0.5\). Consequently, \(H_2^c(t)=0.5/100.5\approx 0.005\), approximately one-thirteenth of the previous value. Using the median rather than the mean in the denominator reduces sensitivity to extreme observations, although it does not eliminate it. The indicator also responds to rapid monotonic trends, as the increasing numerical example illustrates, and is not exclusively a measure of oscillations.

Once \(H_i^c\) has been interpreted, an alarm rule is defined to evaluate the temporal correspondence between the indicator and turbulence onsets. An alarm is recorded on day \(t\) when the indicator crosses above the quantile of order \(q\), calculated over the preceding sixty days. An episode is considered captured if the alarm occurs between its onset day and the following two days; latency corresponds to the number of days elapsed until this first signal. An alarm issued on a calm day is considered false if no episode begins during the following three days, and its frequency is reported in annualized terms.

%------------------------------------

\section{Problem structure}

%\AP{En esta sección también vendría bien poner algún ejemplo}
%\NP{Añadido despues de la definición de Fee}

Unlike the generic time series considered in the previous section, we now present the system under analysis: the liquidity pool of an automated market maker (AMM).

This mechanism---which executes trades against its own reserves according to a deterministic rule---is studied from the perspective of the agent who supports it. An agent who deposits capital in the pool earns fees for each transaction processed but remains exposed to price volatility; therefore, income arises from activity, whereas losses arise from price movements. The research problem addressed in this study consists of understanding how these two flows are related across temporal scales, subject to a causality constraint on the measurements.

\begin{definition}[Constant-product pool]
\label{def:pool-producto-constante}
Let \(x_t>0\) and \(y_t>0\) be the reserves of two assets in a Uniswap V2-type pool. In the absence of fees, a transaction that changes the reserves from \((x_t,y_t)\) to \((x_t+\Delta x,y_t-\Delta y)\) must satisfy the relation
\[
(x_t+\Delta x)(y_t-\Delta y)=x_ty_t
\]

The state of the pool is characterised by the invariant
\[
k_t=x_ty_t,
\]
while its price is defined as the ratio
\[
P_t=\frac{y_t}{x_t},
\]
expressed in units of asset \(y\).

Each transaction also incurs a fee proportional to the traded volume, which accumulates directly in the reserves. Therefore, in the absence of liquidity additions or withdrawals, the invariant \(k_t\) increases with each transaction \cite{adams2020uniswapv2}.

Furthermore, the liquidity provider receives a fixed fraction of this fee: the coefficient \(\mathit{fee}\), a dimensionless constant established by the protocol. Thus, the fees accumulated during interval \(t\) are defined as
\begin{equation}
\label{eq:fee}
\mathrm{Fee}(t)=\mathit{fee}\,V_t,
\end{equation}
where \(V_t\) represents the volume traded during that interval.\footnote{The total pool fee may be greater than the fraction \(\mathit{fee}\) received by the liquidity provider; when this occurs, the difference corresponds to the protocol. The ratio between the two is specified in the data section as the calibration factor \(c\) used in the volume correction.}
\end{definition}

The joint effect of a transaction on the pool price and the liquidity provider's income can be illustrated directly from Definition~\ref{def:pool-producto-constante}. Consider a user who swaps an amount \(\Delta x>0\) of asset \(x\) with the pool for an amount \(\Delta y>0\) of asset \(y\). In the absence of fees, the constant-product condition implies that
\[
\Delta y
=
y_t-\frac{x_ty_t}{x_t+\Delta x}
=
\frac{y_t\Delta x}{x_t+\Delta x}
>0.
\]
After the transaction, the price of asset \(x\), expressed in units of asset \(y\), is
\[
P_{t+1}
=
\frac{y_t-\Delta y}{x_t+\Delta x}
=
P_t
\left(
\frac{x_t}{x_t+\Delta x}
\right)^2
<
P_t.
\]
Therefore, the trade increases the reserve of asset \(x\), reduces the reserve of asset \(y\), and decreases the relative price of the former.

In the absence of fees, if the amount of asset \(x\) exchanged is equal to \(1\,\%\) of its initial reserve, that is, if \(\Delta x=0.01x_t\), then
\[
P_{t+1}
=
\frac{P_t}{(1.01)^2}
\simeq
0.9803P_t,
\]
and therefore the price decreases by approximately \(1.97\,\%\).

The calculation above isolates the effect of the constant-product curve. When fees are incorporated, the volume of this same transaction produces the fee income \(\mathrm{Fee}(t)\) from Equation~\eqref{eq:fee} for the liquidity provider. This fee remains in the reserves. Consequently, in the absence of liquidity deposits or withdrawals, the post-transaction invariant satisfies
\[
k_{t+1}>k_t.
\]

Thus, a transaction simultaneously affects the two central quantities of the problem: the change in the reserves modifies the pool price, whereas the fee compensates the liquidity provider in proportion to the trading volume.

\subsection{Economics and Time Series of the System}

An agent who deposits capital in the pool contributes both assets in the same proportion as the reserves and owns a fraction of the total deposited value (\(\mathrm{TVL}_t\)). The agent's economic incentive involves two opposing forces: income from accumulated fees (\(\mathrm{Fee}(t)\)) in proportion to their share and a cost imposed by arbitrage in response to price movements. At second order, this loss to arbitrage increases with price volatility and is equivalent to \(\sigma^2/8\) of the pool value per unit of time \cite{milionis2022lvr}.

Because arbitrage keeps the pool price (\(P_t\)) aligned with the external price \cite{angeris2021analysis}, the system directly inherits market volatility. Thus, the pool is described by three time series: price \(P_t\) (external forcing), volume \(V_t\) (which reflects activity and generates fees), and \(\mathrm{TVL}_t\) (deposited capital). These three series become coupled during episodes of high volatility, when arbitrage and capital repositioning intensify the use of the pool, causing the provider's income and costs to tend to coincide over time.

\subsection{Pool Return-to-Risk: The \(m\) Indicator}

The tension between the income flow (fees) and risk exposure (volatility) makes it possible to construct a scalar summary at daily frequency, which we denote by \(m\).

\begin{definition}[Fee yield]
\label{def:rendimiento-comisiones}
For each day \(d\), the fee yield per unit of deposited capital is defined as
\[
P_s(d)
=
\frac{\mathrm{Fee}(d)}{\mathrm{TVL}_d}
=
\mathit{fee}\,\frac{V_d}{\mathrm{TVL}_d},
\]
where
\[
\mathrm{Fee}(d)=\mathit{fee}\,V_d
\]
Here, \(V_d\) represents the daily traded volume, \(\mathit{fee}\) is the provider fee coefficient, and \(\mathrm{TVL}_d\) is the deposited capital.
\end{definition}

\begin{definition}[Realised variance]
\label{def:varianza-realizada}
For each day \(d\), the realised variance is defined as
\[
RV_d^2
=
\sum_{t\in d} r_t^2,
\]
where \(r_t\) represents the intraday logarithmic return of the pool price.
\end{definition}

\begin{definition}[\(m\) indicator]
\label{def:indicador-m}
For each day \(d\), the \(m\) indicator is defined as
\[
m(d)
=
\frac{P_s(d)}{RV_d^2}
\]
\end{definition}

This quantity represents the daily fee yield relative to its dominant risk. Operationally, except for the constant \(\mathit{fee}\), it is equivalent to dividing capital turnover, \(V_d/\mathrm{TVL}_d\), by realised variance. Since fees are not observed directly, they are inferred from this identity using the two measured series.

This ratio does not incorporate any prior smoothing: each day has its own value, and any temporal scale present in the analysis arises exclusively from the multiresolution decomposition. Within this framework, the approximations \(A_j(d)\) represent the smoothed versions of \(m(d)\) at each scale, whose exact reconstruction is given by
\[
m(d)
=
A_J(d)
+
\sum_{j=1}^{J}D_j(d)
\]

Furthermore, since the daily loss to arbitrage is equivalent to \(RV_d^2/8\) per unit of pool value, the coverage condition is
\[
P_s(d)\geq\frac{RV_d^2}{8}
\quad\Longleftrightarrow\quad
m(d)\geq\frac{1}{8}.
\]

Here, \(m=1/8\) represents the point at which the cost of arbitrage is covered. Since the theoretical rate \(\sigma^2/8\) assumes a frictionless setting and constitutes an upper bound on the actual cost, the condition \(m(d)\geq 1/8\) is sufficient to ensure such coverage. Under this framework, the operational threshold \(m(d)>0.4\) used in the analysis is approximately three times this reference level.

\subsection{Problem Formulation}

Up to this point, the calculated variables are aggregated: the overall correlation or the daily value of \(m\) reduces dynamics from different temporal scales to a single observation. By averaging the overall behaviour, these aggregates may conceal relationships within the system if they operate only at specific scales.

For this reason, three lines of analysis are formulated from the system developed above using the tools introduced in the previous section. The first two break down the central problem: they analyse the relationship between activity and volatility across temporal scales and decode the information carried by the scale components of the indicator that summarises them. The third evaluates the causality constraint adopted to determine, wherever a comparison is possible, whether it is feasible to require the measurement to depend exclusively on the available historical information.

\subsubsection{Scale Structure of the Coupling}

To address this aspect, \(\rho_B(t)\) is defined as the band-energy correlation between volume and the absolute value of the logarithmic price return over a causal window. The analysis seeks to determine whether the coupling between activity and volatility is concentrated at specific scales and whether this relationship is reorganised around transitions to the turbulent regime.

\subsubsection{Scale Composition of \(m\)}

Considering the causal decomposition of the \(m\) indicator introduced previously, a high value of \(m\) indicates that the pool generates returns efficiently per unit of risk; however, it remains to be determined whether the temporal scale from which that value originates is decisive. Specifically, two questions are posed:

\begin{enumerate}
    \item[(i)] When \(m\) crosses its operational threshold, which scale components distinguish persistent signals from short-lived ones, and can this persistence be identified from their scale of origin?

    \item[(ii)] Does the composition of \(m\) contain information about return relative to risk on subsequent days that is not contained in the aggregate level of the indicator?
\end{enumerate}

To answer the second question, it is necessary to define the performance over the following \(k\) days. For a given day \(d\) and a horizon of \(k\) days, this performance is defined as
\[
m_f(d,k)
=
\frac{\sum_{u=d+1}^{d+k}P_s(u)}{\sum_{u=d+1}^{d+k}RV_u^{2}}.
\]
That is, \(m_f(d,k)\) is the sum of the fee returns over the interval divided by the total accumulated variance. This calculation is equivalent to averaging the daily values of \(m\), weighting each day by the risk actually borne by the pool.

\subsubsection{Feasibility of the Causal Construction}

Under the requirement of operating exclusively with past information, the effectiveness of the proposed multiresolution indicator is evaluated. Specifically, when detecting transitions in the system, the analysis examines whether the performance of this metric can match the capability of a simple benchmark estimator under the same causal protocol.

%------------------------------------
\section{Data description}

\subsection{Time Series and Pool Panel}

The system described in the previous section is analysed using a panel of 18 Uniswap V2-type pools, sampled every five minutes over two years (from 31 March 2024 to 30 March 2026; three of these pools have only partial histories). The panel, as shown in Table~\ref{tab:pools-panel}, comprises ten stablecoin/volatile pools, for which the pool price is equivalent to the dollar price of the token, and eight volatile/volatile pools, for which the price represents the exchange ratio between the two tokens\footnote{For volatile/volatile pairs, WETH is adopted as the reference asset; thus, the pool price is consistently expressed in units of the second token per unit of WETH.}. Based on events recorded directly on the blockchain\footnote{Swap data are obtained from the public event logs of the corresponding networks through customised queries in Dune Analytics (\url{https://dune.com}).}, the three time series required for the study are constructed: price (calculated as the ratio of the pool reserves), trading volume in each interval, and total capital deposited (\textit{Total Value Locked}, or TVL). For the low-frequency analysis, these series are aggregated at the end of each day.

\begin{table}[!htbp]
\centering
\caption{Liquidity pools included in the panel}
\label{tab:pools-panel}
\resizebox{\columnwidth}{!}{%
\begin{tabular}{llllrl}
\hline
Pair & Type & DEX & Chain & LP fee & Days of history \\
\hline
WETH/USDC   & S/V & Uniswap     & Ethereum & 0.25\% & 730 \\
WETH/USDT   & S/V & Uniswap     & Ethereum & 0.25\% & 730 \\
DAI/WETH    & S/V & Uniswap     & Ethereum & 0.25\% & 730 \\
USDS/SKY    & S/V & Uniswap     & Ethereum & 0.25\% & 556 (from 2024-09-21) \\
WETH/USDC   & S/V & Uniswap     & Base     & 0.25\% & 730 \\
WBNB/USDT   & S/V & PancakeSwap & BNB      & 0.17\% & 730 \\
ETH/USDT    & S/V & PancakeSwap & BNB      & 0.17\% & 730 \\
CAKE/USDT   & S/V & PancakeSwap & BNB      & 0.17\% & 730 \\
ARK/USDT    & S/V & PancakeSwap & BNB      & 0.17\% & 217 (from 2025-08-26) \\
WPOL/USDC   & S/V & Quickswap   & Polygon  & 0.30\% & 730 \\
WBTC/WETH   & V/V & Uniswap     & Ethereum & 0.25\% & 730 \\
MKR/WETH    & V/V & Uniswap     & Ethereum & 0.25\% & 730 \\
UNI/WETH    & V/V & Uniswap     & Ethereum & 0.25\% & 730 \\
PEPE/WETH   & V/V & Uniswap     & Ethereum & 0.25\% & 730 \\
FLOKI/WETH  & V/V & Uniswap     & Ethereum & 0.25\% & 730 \\
BANANA/WETH & V/V & Uniswap     & Ethereum & 0.25\% & 730 \\
SPX/WETH    & V/V & Uniswap     & Ethereum & 0.25\% & 730 \\
NEIRO/WETH  & V/V & Uniswap     & Ethereum & 0.25\% & 612 (from 2024-07-27) \\
\hline
\end{tabular}%
}

\medskip

\parbox{\columnwidth}{%
\footnotesize
\textit{Note:} The full period extends from 31 March 2024 to 30 March 2026, corresponding to 730 days and \(210\,240\) five-minute intervals. S/V denotes stablecoin/volatile pools, and V/V denotes volatile/volatile pools. The LP fee corresponds to the fraction of the fee received by the liquidity provider.
}
\end{table}

Coupling and detection are evaluated through an event study around turbulence onsets, which requires recurrent episodes that are sufficiently separated in time. This condition is met by seven pools.\footnote{The selected pools are \textsc{wbnb/usdt}, \textsc{eth/usdt}, \textsc{cake/usdt}, \textsc{wpol/usdc}, \textsc{uni/weth}, \textsc{pepe/weth}, and \textsc{neiro/weth}.} For the remaining analyses, the sample is adjusted to the requirements of each metric: the scale composition of \(m\) excludes \textsc{ark/usdt}, whose lack of intraday transactions distorts its realised variance, and uses 17 pools, whereas the persistence study retains the full sample of 18.

\subsection{Scales in Time Units}

Given a five-minute sampling grid for price and a daily sampling grid for the indicator, the methodological scale rule translates into specific time intervals, as shown in Table~\ref{tab:escalas-temporales}.

\begin{table}[H]
\centering
\caption{Correspondence between scales and time units}
\label{tab:escalas-temporales}
\resizebox{\columnwidth}{!}{%
\begin{tabular}{lll}
\hline
Level & Daily series & Five-minute price \\
\hline
\(D_1\) & 2--4 days          & 10--20 minutes \\
\(D_2\) & 4--8 days          & 20--40 minutes \\
\(D_3\) & 8--16 days         & 40--80 minutes \\
\(D_4\) & 16--32 days        & 1.3--2.7 hours \\
\(D_5\) & 32--64 days        & 2.7--5.3 hours \\
\(D_6\) & 64--128 days       & 5.3--10.7 hours \\
\(A_6\) & More than 128 days & More than 10.7 hours \\
\hline
\end{tabular}%
}
\end{table}

\subsection{Data Processing}

A uniform procedure is applied to all the pools analysed. First, the initial segment preceding the first recorded transaction is truncated, and days containing fewer than half of the expected intervals are discarded. The statistical characterisation is then performed using robust estimators, specifically the median and the median absolute deviation.

Unlike price and \(\mathrm{TVL}_t\), which are read directly from the reserves in the pool state, dollar-denominated volume is reconstructed externally. The data provider calculates it from the swap events emitted by the smart contract, converting them into dollars using price references external to the pool.

As established in the previous section, the state of the pool is characterised by the invariant \(k\). In the absence of liquidity additions or withdrawals, this invariant increases because of the fees accumulated in each transaction \cite{adams2020uniswapv2}. This increase makes it possible to relate the growth of the invariant to the fees generated, \(\mathrm{Fee}(t)\), and \(\mathrm{TVL}_t\) through the approximation
\[
\Delta\ln k_t
\approx
2c\frac{\mathrm{Fee}(t)}{\mathrm{TVL}_t},
\]
where \(c\) is a calibration factor estimated for each pool, defined as the ratio of the total fee to the portion allocated to the liquidity provider.

Under this accounting criterion, the correction is activated when the observed growth of the invariant is significantly lower than its expected theoretical growth, provided that the reserves remain stable. Formally, the rule is applied when the following conditions are satisfied simultaneously:
\[
\begin{gathered}
\Delta\ln k_t
<
\theta
\left(
2c\frac{\mathrm{Fee}(t)}{\mathrm{TVL}_t}
\right),
\qquad
\frac{V_t}{\mathrm{TVL}_t}>\alpha,
\\[4pt]
\frac{\left|\Delta\mathrm{TVL}_t\right|}{\mathrm{TVL}_t}<\beta,
\end{gathered}
\]
where \(\theta\in(0,1)\) determines the minimum accepted fraction of the expected growth of the invariant, \(\alpha>0\) requires the volume to be sufficient for the accounting signal to be observable, and \(\beta>0\) limits variations in \(\mathrm{TVL}_t\) that could correspond to liquidity additions or withdrawals. The operational values used in the implementation are
\[
\theta=0.3,
\qquad
\alpha=\beta=0.02
\]

When these conditions are satisfied, the reconstructed volume \(V_t\) is replaced by the implicit volume supported by the observed growth of the invariant. This correction was applied to 50 intervals across 6 of the 18 pools, with no material effect on the results.

%------------------------------------
\section{Results}

To answer the questions posed in the previous section, we present the results of the analysis in the following order. First, the original $m$ signal is presented together with its scale components, this decomposition provides the basis for the first findings of the study, which link the persistence of the pool's expected return relative to its risk to the temporal scale from which the signal originates.

Second, the scale-by-scale coupling between a pool's activity and volatility around the event is presented, followed by the viability of the causal construction using a fluctuation-to-trend ratio indicator.

Before examining the detailed results, the labeling of regime-change events in a representative pool from the panel is introduced.

\subsection{Volatility regime and event labeling}

These events represent the regime change to a turbulent state and provide the basis for the analysis and results presented below.

\begin{figure}[htbp]
\centering
\includegraphics[width=\columnwidth]{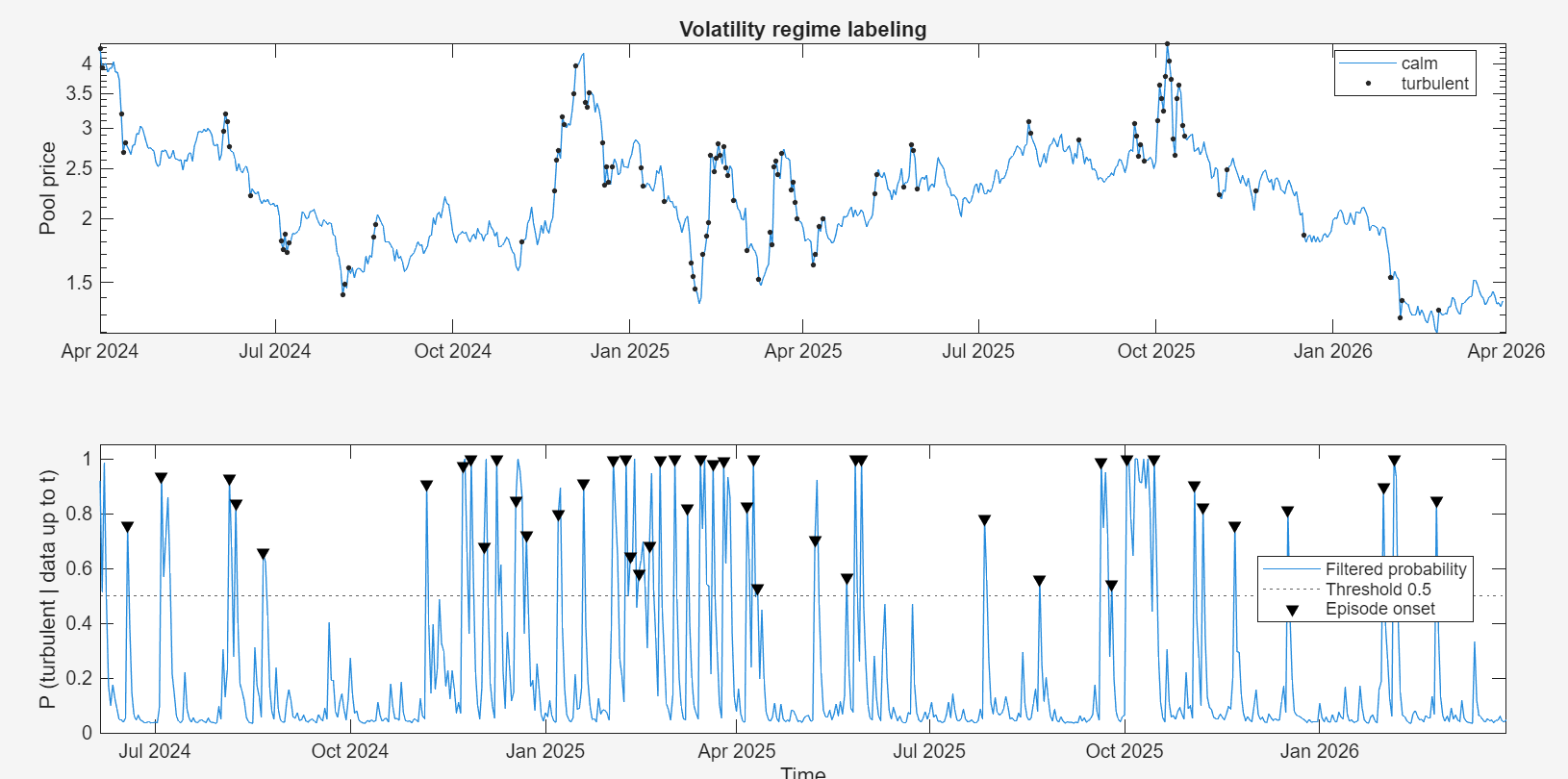}
\footnotesize
\caption{\footnotesize Volatility regime labeling. CAKE/USDT pool (PancakeSwap V2, BNB Chain), between April 1, 2024 and March 30, 2026: 106 turbulent days (15\%) and 47 episode onsets.}
\label{fig:regimen}
\end{figure}

In Figure~\ref{fig:regimen}, regime identification combines two visual signals. In the upper panel, points are displayed on the pool price series during periods of financial stress. The lower panel, in turn, shows the filtered probability of turbulence, which acts as a change indicator; when it exceeds the level of $0.5$ (marked with triangles), the regime change in the pool is formally confirmed.

Under this classification, the identified states show that annualized volatility ranges between $24\,\%$ and $57\,\%$ in the calm state and increases to between $72\,\%$ and $220\,\%$ during periods of turbulence.

To structure the model, we proposed classifying the regimes using price volatility rather than activity, assuming that turbulence is driven by the asset (it is systemic), whereas volume is specific to each pool (it is local). The data validate this choice: when analyzing pools that share the same underlying asset, price volatility exhibits a correlation of $0.92$, in contrast to the $0.38$ obtained for transaction volume.

\subsection{Indicator $m$: Signal persistence according to its scale of origin}

The temporal decomposition of the indicator reveals an uneven distribution of variability. The three fastest scales (from two to sixteen days) account for $63\%$ of the total variance ($31\%$ in the 2--4-day band, $20\%$ in the 4--8-day band, and $12\%$ in the 8--16-day band), which characterizes \(m\) as a series dominated by high-frequency dynamics. Nevertheless, the background component interrupts this decline and contributes $14\%$, showing that the slow dynamics of the indicator constitute a distinct and non-negligible component.

\begin{figure}[htbp]
\centering
\includegraphics[width=\columnwidth]{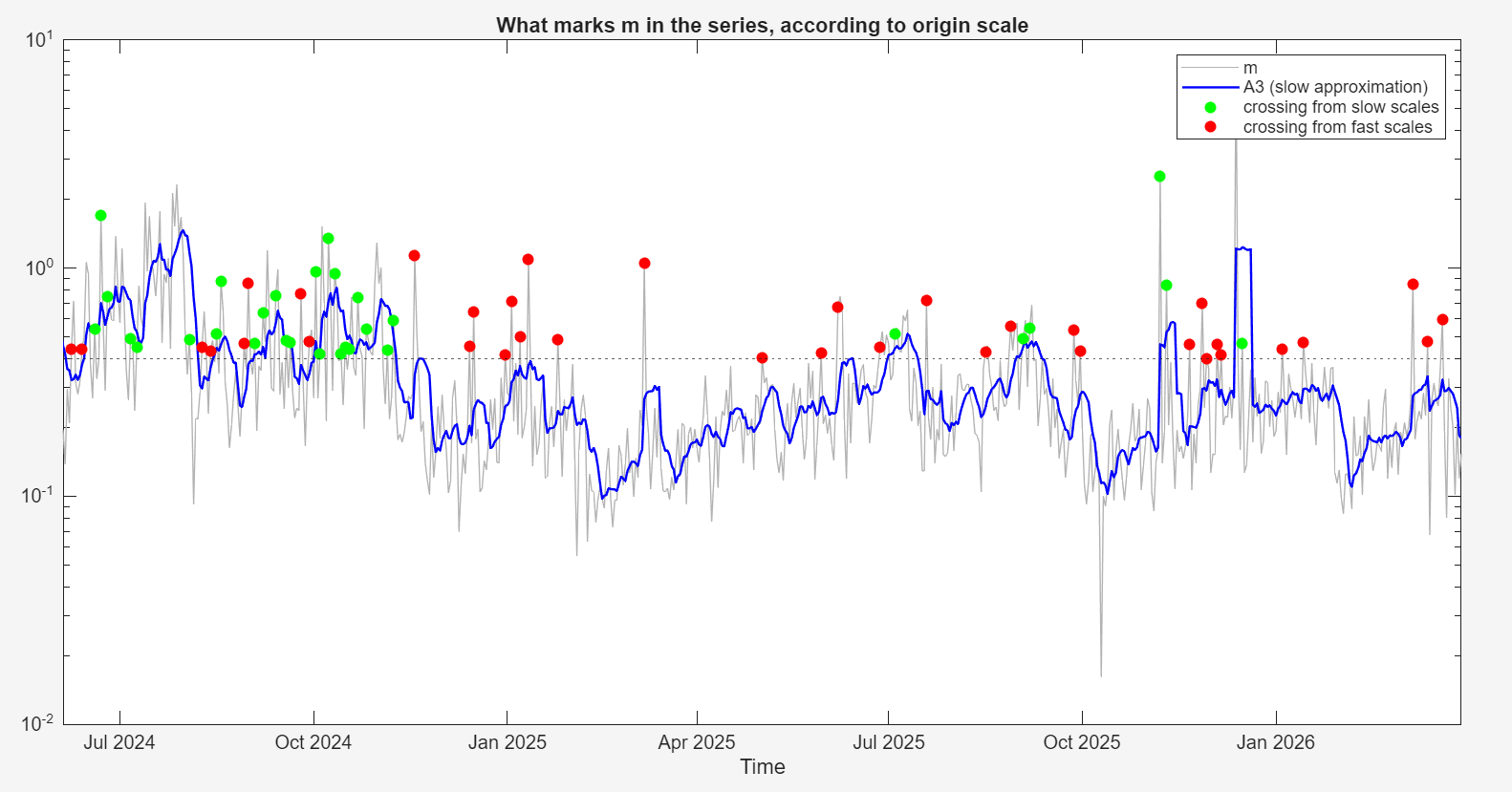}

\caption{\footnotesize Signals of indicator $m$ classified according to their scale of origin. The points indicate upward crossings: green points denote those of slow origin, for which $A_3$ also exceeds the threshold, and red points denote those arising only from recent movement. CAKE/USDT pool (PancakeSwap V2, BNB Chain)}
\label{fig:m-origen}
\end{figure}

This coexistence of rapid fluctuations against a slow-moving background implies that an extreme value may arise from a transitory disturbance or from a sustained elevated level. To distinguish between these two dynamics without distorting the original series, the classification criterion illustrated in Figure~\ref{fig:m-origen} is implemented, where a signal is defined as the day on which the indicator \(m\) crosses above the operating threshold \(m(d)>0.4\), and its duration corresponds to the number of consecutive days for which it remains above this threshold.

Both the crossing and the duration are calculated from the undecomposed series; the decomposition is used to classify the time of the crossing according to whether the average over the previous eight days, the approximation $A_3$, also exceeds the threshold. In that case, the signal is classified as having a slow origin. Thus, Table~\ref{tab:persistence} reports longer durations among signals classified by an elevated eight-day average. This association may partly reflect the persistence already encoded in the classification rule; an autoregressive or moving-average benchmark is needed to assess incremental information.

The \(0.4\) threshold is an exogenous operational choice used to define the crossings and to classify them by origin; the classification rule does not depend on this specific level. Robustness checks using thresholds of \(0.3\) and \(0.5\) yield identical conclusions: median durations of two days versus one, with log-rank \(p\)-values ranging from \(1.6 \times 10^{-38}\) to \(4.9 \times 10^{-33}\).

\begin{table}[H]
\centering
\caption{Duration of \(m\) signals by scale of origin}
\label{tab:persistence}

\resizebox{\columnwidth}{!}{%
\begin{tabular}{lrr}
\hline
Origin of the crossing & Signals & Median duration (days) \\
\hline
Recent fluctuations only (\(A_3\leq 0.4\))       & 636 & 1 \\
Supported by the slow component (\(A_3>0.4\))   & 490 & 2 \\
\hline
\end{tabular}%
}

\medskip

\parbox{\columnwidth}{%
\footnotesize
\textit{Note:} Panel of 18 pools. Duration is measured as the number of consecutive days for which \(m\) remains above the threshold, with right-censoring at the end of the sample. The log-rank test yields \(p=9.7\cdot 10^{-35}\).
}
\end{table}

\subsection{Indicator $m$: Information on future performance}

The same value of \(m\) may arise from a transitory one-day fluctuation or from a level sustained over several weeks. This section evaluates whether, once the level of the indicator is known, this difference in origin provides information about the pool's subsequent performance.

To answer this question, the variable \(m_f\), defined in Section~5.3.2, is used as the comparison variable. This variable captures the pool's return relative to risk over the following \(k\) days by dividing the total fees over the interval by the total accumulated variance. Equivalently, it corresponds to an average of the daily values of \(m\), weighting each day by the risk actually borne by the pool.

The results in Table~\ref{tab:rendimiento_m} show opposite signs of partial rank association at fast and slow scales. For each horizon \(k\), the median partial rank correlation increases monotonically across the seven bands. In particular, for \(k=14\), the \(D_1\) band, corresponding to a period of 2--4 days, has partial rank correlation \(-0.134\), whereas the long-term component \(A_6\) has partial rank correlation \(+0.138\); the sign changes between the 16--32-day and 32--64-day bands. Thus, conditional on the aggregate indicator level, short-scale fluctuations are negatively associated with subsequent performance, whereas slower components are positively associated with it; no causal effect is identified. This structure is consistent: the increasing ordering of the scales remains identical across the three time horizons and is reproduced individually in the 17 pools analyzed.

\begin{table}[H]
\centering
\caption{Information provided by each component of \(m\) about future return-to-risk}
\label{tab:rendimiento_m}
\resizebox{\columnwidth}{!}{%
\begin{tabular}{lrrrc}
\hline
Component & \(k=7\) & \(k=14\) & \(k=30\) & Pools \(-\) / \(+\) \\
\hline
\(D_1\) (2--4 days)      & \(-0.179\) & \(-0.134\) & \(-0.104\) & 17 / 0 \\
\(D_2\) (4--8 days)      & \(-0.100\) & \(-0.077\) & \(-0.057\) & 17 / 0 \\
\(D_3\) (8--16 days)     & \(-0.025\) & \(-0.029\) & \(-0.008\) & 12 / 5 \\
\(D_4\) (16--32 days)    & \(-0.007\) & \(-0.001\) & \(+0.006\) & 9 / 8 \\
\(D_5\) (32--64 days)    & \(+0.041\) & \(+0.045\) & \(+0.066\) & 6 / 11 \\
\(D_6\) (64--128 days)   & \(+0.144\) & \(+0.108\) & \(+0.090\) & 0 / 17 \\
\(A_6\) (more than 128 days) & \(+0.153\) & \(+0.138\) & \(+0.049\) & 3 / 14 \\
\hline
\end{tabular}%
}

\medskip

\parbox{\columnwidth}{%
\footnotesize
\textit{Note:} The values correspond to the partial rank correlation between each component and \(m_f\), controlling for the daily level \(m(d)\); the median across the 17 pools is reported. For \(k=14\), the final column indicates the number of pools in which the contribution of each scale is negative or positive.
}
\end{table}

Taken together, the three fastest bands account for 63\% of the variance of \(m\), but exhibit negative partial rank associations with subsequent performance. Conditional on the same indicator level, fluctuations at scales shorter than eight days are associated with lower performance over the following period, whereas scales longer than sixty-four days make a positive contribution. Despite accounting for only 14\% of the reported variance, the long-term component has one of the strongest positive partial rank associations. Since the components sum algebraically to \(m(d)\), these individual partial correlations are not additive contributions to predictive performance; an out-of-sample comparison against an aggregate-only model would be needed to establish incremental predictive value.

\subsection{Scale structure of the coupling}

To study the coupling between volume and volatility, we use \(\rho_B(t)\) (Definition~\ref{def:correlacion-energias-banda}), with \(x\) representing the pool volume and \(y\) the absolute value of the logarithmic return of its price. Local energies are calculated using \(w=288\) observations, corresponding to one day, and the correlation is calculated over a window of \(W=30\) days. The decomposition is carried out up to \(J=13\), and the undecomposed series is also evaluated.

\begin{figure}[H]
\centering
\includegraphics[width=\columnwidth]{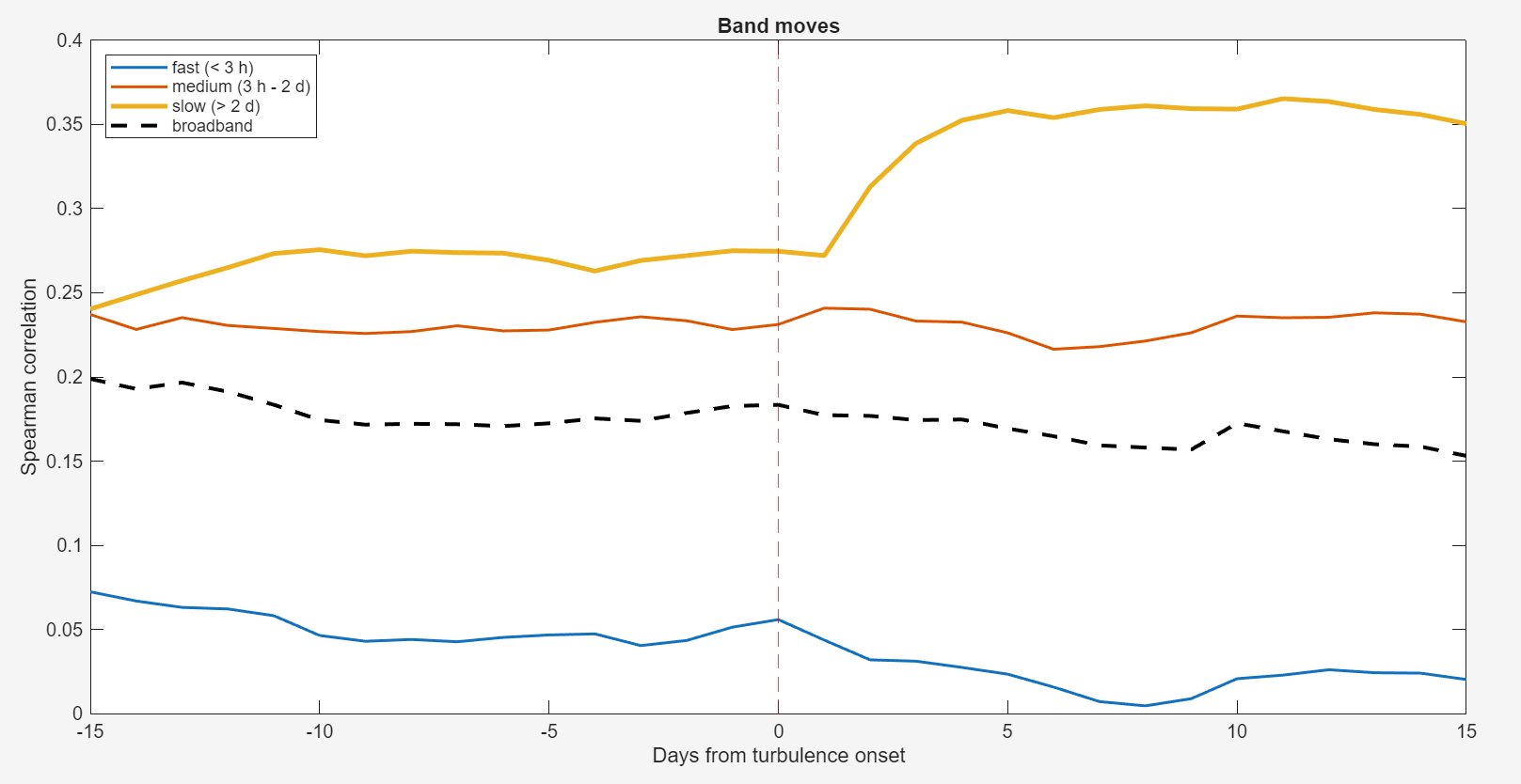}
\caption{\footnotesize Coupling between volume and volatility around the onset of turbulence episodes.}
\label{fig:coupling}
\end{figure}

Using these correlations, we conduct an event study using a $\pm 15$-day event window centered on the 146 turbulence onsets detected in the panel of seven pools (Figure~\ref{fig:coupling}). These onsets are considered isolated because no other onset occurred during the preceding ten days. The regime and the coupling are calculated separately: the labeling provides the onset dates without entering into the calculation of the correlation series.

\begin{table}[htbp]
\centering
\caption{Band-specific coupling around the onset of turbulence}
\label{tab:coupling}
\resizebox{\columnwidth}{!}{%
\begin{tabular}{lrrr}
\hline
Band & before \([-7,-1]\) & after \([+1,+7]\) & after \([+8,+15]\) \\
\hline
Fast (\(D_1\)--\(D_4\), period \(<\) 3 h)   & \(+0.05\) & \(+0.03\) & \(+0.02\) \\
Medium (\(D_5\)--\(D_8\), 3 h--2 d)         & \(+0.23\) & \(+0.23\) & \(+0.23\) \\
Slow (\(D_9\)--\(D_{13}\), \(>\) 2 d)       & \(+0.27\) & \(+0.34\) & \(+0.36\) \\
Undecomposed                                & \(+0.18\) & \(+0.17\) & \(+0.16\) \\
\hline
\end{tabular}%
}

\medskip

\parbox{\columnwidth}{%
\footnotesize
\textit{Note:} Spearman correlation \(\rho_B(t)\). The values are correlations,
bounded between \(-1\) and \(1\); a value of zero indicates no association between
activity and volatility in that band.
}
\end{table}

As shown in Table~\ref{tab:coupling}, the coupling increases with scale even far from the events: the slow band starts at a level five times that of the fast band. The observed increase is concentrated in the slow band: its correlation rises by $0.09$ and continues to rise through day $+15$, whereas the reported fast and medium correlations change little. Statistical significance of these event-time differences has not been established without dependence-aware confidence intervals.

\subsection{Detection capability of the multiresolution indicator \(H_i^c\)}

In this final part, we analyze the ability of the causal indicator \(H_i^c\) (Definition~\ref{def:indicador-causal-hc}) to detect regime transitions by comparing it with a simple estimator evaluated under the same assessment scheme, in order to isolate the contribution of the multiresolution decomposition.
The simple estimator is the daily median of the absolute five-minute log returns of the pool price---a robust measure of intraday volatility that requires no decomposition. Likewise, the daily value of \(H_2^c\) is the median of the indicator computed from the five-minute series. Both are evaluated using the alarm rule with \(q=0.90\).
Figure~\ref{fig:indicador-causal-h} shows the application of the indicator to one of the pools in the panel. The upper panel shows the evolution of the price, whereas the lower panel displays the indicator \(H_2^c\) at a daily frequency. Increases in the indicator reflect periods in which fluctuations gain greater weight relative to the trend.

\begin{figure}[htbp]
\centering
\includegraphics[width=\columnwidth]{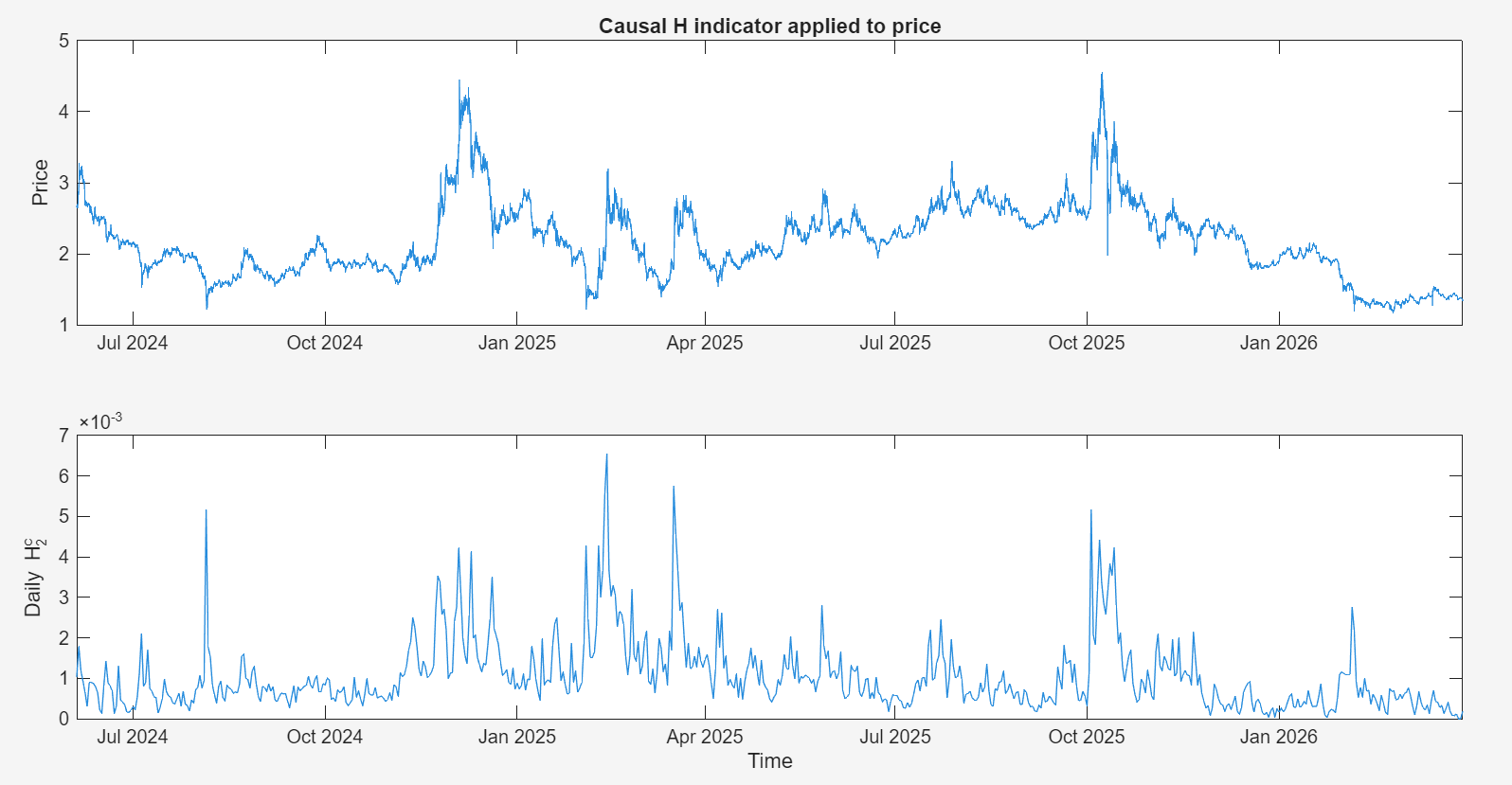}
\caption{\footnotesize Evolution of the price and the indicator \(H_2^c\) in one of the pools in the panel.}
\label{fig:indicador-causal-h}
\end{figure}

When the indicator is applied to the pool price with a depth of \(i=2\), \(H_2^c\) captures \(42\%\) of turbulence onsets with a median latency of zero days; that is, the alarm is triggered on the same day that the episode begins. The simple estimator identifies the same proportion of episodes but has a median latency of one day. As a trade-off, the simple estimator records one fewer false alarm per year: \(9.5\), compared with \(10.5\) for the multiresolution indicator.

%------------------------------------
\section{Conclusions}

This study develops a past-looking multiresolution decomposition for liquidity-pool time series in decentralized finance. Through this approach, a central statement regarding the temporal structure of these markets is formulated: pool variability is concentrated in the fast part of its spectrum, whereas its structure—understood as the relationship between series, signal persistence, and predictive information about subsequent performance—resides in the slow part. These two dimensions are indistinguishable under traditional aggregate metrics, such as means or global variances, because calculations over the complete series merge them before they can be compared.

The wavelet features can be calculated from historical observations at each date. A fully real-time interpretation of the regime-labeling exercise additionally requires estimation of all Markov-model parameters using only information available at that date; otherwise the event labels remain retrospective. On this basis, applying the method to the pools analyzed yields three main findings.

The first finding shows that, for the same aggregate indicator level, its scale composition provides information about subsequent performance. At the fourteen-day horizon, the partial rank correlation ranges from \(-0.134\) in \(D_1\) to \(+0.138\) in \(A_6\), with a sign change between the 16--32-day and 32--64-day bands. This ordering is maintained across the three horizons evaluated. The second finding concerns the structure of the coupling between activity and volatility: following regime transitions, the correlation of the slow band increases from \(0.27\) to \(0.36\), whereas the correlation of the undecomposed series remains around \(0.18\). The third finding supports the viability of the causal construction, which matches the detection capability of a simple estimator and identifies the onset of episodes on the same day, compared with the one-day delay of the benchmark estimator.

The magnitude of the effect is moderate and should be characterized accordingly: for observations with the same indicator level, the difference in performance between extreme compositions is approximately \(12\%\) over a fourteen-day horizon. Moreover, the decomposition does not improve the predictive capability of a simple smoothing procedure or the detection power of an elementary robust estimator. Its contribution is descriptive: it identifies the nominal scales at which the observed associations are strongest, without establishing causal mechanisms or incremental out-of-sample predictive performance.

The scope of this study is limited to eighteen constant-product pools observed over two years; its results are strictly associational, and uncertainty quantification for overlapping-window correlations and sensitivity to the choice of decomposition depth remain to be established. The most natural extension concerns temporal resolution: turbulence episodes are short-lived, and the daily frequency limits the precision with which a detector can locate them. Therefore, adapting the procedure to an intraday scale represents the direction with the greatest potential for improvement.

%% TODO-NP: nota de disponibilidad del código (los dos scripts van como ancillary files de arXiv en anc/, o a un repositorio). Una frase antes de la bibliografía, p. ej. en un párrafo sin numerar «Code availability». Conservar de la intro del apéndice lo de MATLAB R2025a, solo funciones base y que se evita modwtmra por no ser causal.

\bibliographystyle{amsplain}
\bibliography{bibliography}

\end{document}